\newif\iffundamenta
\fundamentafalse 
\iffundamenta
\documentclass[12pt]{amsart}
\else
\documentclass[11pt,a4paper]{scrartcl}
\fi

\usepackage{amsmath,amssymb,amsthm}
\usepackage{mathtools}
\usepackage{enumitem}
\usepackage{hyperref}
\usepackage[capitalize,noabbrev,nameinlink]{cleveref}

\usepackage{tikz}
\usetikzlibrary{cd,shapes.geometric}

\usepackage{mydraw}
\usepackage{knotpics2}

\iffundamenta
\fi

\iffundamenta
\makeatletter
\@namedef{subjclassname@2020}{%
  \textup{2020} Mathematics Subject Classification}
\makeatother
\fi

\theoremstyle{plain}
\newtheorem{theorem}{Theorem}

\newtheorem{proposition}[theorem]{Proposition}

\newtheorem{lemma}[theorem]{Lemma}

\theoremstyle{definition}
\newtheorem{definition}[theorem]{Definition}

\theoremstyle{remark}

\newtheorem{remark}[theorem]{Remark}

\newcommand\boundin{\partial_{\mathsf{in}}}
\newcommand\boundout{\partial_{\mathsf{out}}}

\iffundamenta
\title[Twist of Frobenius algebra and link homology]{Remark on twists of Frobenius algebra and link homology}
\author[N.~Ito]{Noboru Ito}
\address{Shinshu University}
\email{nito@shinshu-u.ac.jp}

\author[K.~Nakagane]{Keita Nakagane}
\email{keita.nakagane@gmail.com}

\author[J.~Yoshida]{Jun Yoshida}
\address{RIKEN AIP}
\email{jun.yoshida@riken.jp}

\date{August~19, 2025}



\else
\title{Remark on twists of Frobenius algebra and link homology}
\author{Noboru Ito, Keita Nakagane, and Jun Yoshida}
\date{September~9, 2025} 
\fi

\iffundamenta
\fi

\begin{document}

\maketitle

\section{Introduction}\label{sec:Introduction}

Khovanov~\cite{Khovanov2000} defined a link invariant, which is nowadays known as \emph{Khovanov homology}.
Later, in \cite{BarNatan2005} and \cite{Khovanov2006}, it was suggested that there are some interesting variants of Khovanov homology indexed by Frobenius algebras.
Namely, for a link diagram $D$ and a commutative Frobenius algebra $A$, define a commutative cube $V_{D;A}$ by
\begin{equation*}
V_{D;A}(S)\coloneqq\bigotimes_{C\in\pi_0(D_S)} A_C
\end{equation*}
for each state $S$ on $D$, where $A_C$ is a copy of $A$ associated to a circle component of the smoothing $D_S$.
On the other hand, if $\theta\in A$ is an invertible element, then one can \emph{twist} the coalgebra structure of $A$ by $\theta$ as
\begin{equation*}
\varepsilon^\theta(x)\coloneqq \varepsilon(\theta x)
,\qquad
\Delta^\theta(x)\coloneqq\Delta(\theta^{-1}x)
.
\end{equation*}
We denote by $A^\theta$ the resulting Frobenius algebra.
In the paper \cite{Khovanov2006}, it is asserted that this twisted Frobenius algebra $A^\theta$ results in the same link homology as $A$, and this assertion is widely accepted for nearly two decades.
Khovanov however only gave a sketch of the proof.
We communicated with Khovanov, and it turned out that the result requires more careful and detailed proof.

We here elaborate the gap in the original construction.
In order to construct an isomorphism $V_{D;A^\theta}\cong V_{D;A}$, a key ingredient is a function $\nu:\coprod_{S}\pi_0(D_S)\to\mathbb Z$ such that the following square commutes for each state $S$ and each crossing $c\notin S$:
\begin{equation*}
\begin{tikzcd}[column sep=10ex]
V_{D;A^\theta}(S) \ar[r,"{\bigotimes_C \theta^{\nu(C)}}"] \ar[d,"\xi_c"'] & V_{D;A}(S) \ar[d,"\xi_c"] \\
V_{D;A^\theta}(S\cup\{c\}) \ar[r,"{\bigotimes_{C'}\theta^{\nu(C')}}"] & V_{D}(S\cup\{c\})
\end{tikzcd}
\end{equation*}
where $\xi_c$ is the structure morphism of the commutative cubes.
Since $\theta$ is invertible, it is obvious that the morphisms form an isomorphism of commutative cubes.
According to the original proof in \cite{Khovanov2006}, it is asserted that one can construct such a function $\nu$ on $D$ by induction on the cardinality $\lvert S\rvert$.
In each induction step, we need to make a choice of how we distribute the value $\nu(C)$ whenever $C$ splits into two in a new state.
The source of the problem is that such a choice cannot be arbitrary.

We can see the problem more explicitly in a concrete example; take $D$ to be the diagram for the right-handed trefoil in \cref{fig:Trefoil}.%
\begin{figure}
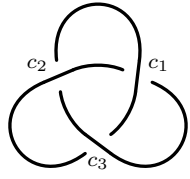

\setlength\belowdisplayskip{0pt}
\setlength\belowdisplayshortskip{0pt}
\begin{equation*}
\diagKnotTrefoil*[c_1,c_2,c_3]{positive,positive,positive}
\end{equation*}
\caption{The right-handed trefoil}
\label{fig:Trefoil}
\end{figure}
In this example, there is an unfortunate choice in the induction steps which causes deadlock in later steps.
Namely, \cref{fig:UnfortunateTwistingWeight} depicts the list of smoothings $D_S$ for the states $S$, where the red label attached to each circle $C$ indicates the value $\nu(C)$.%
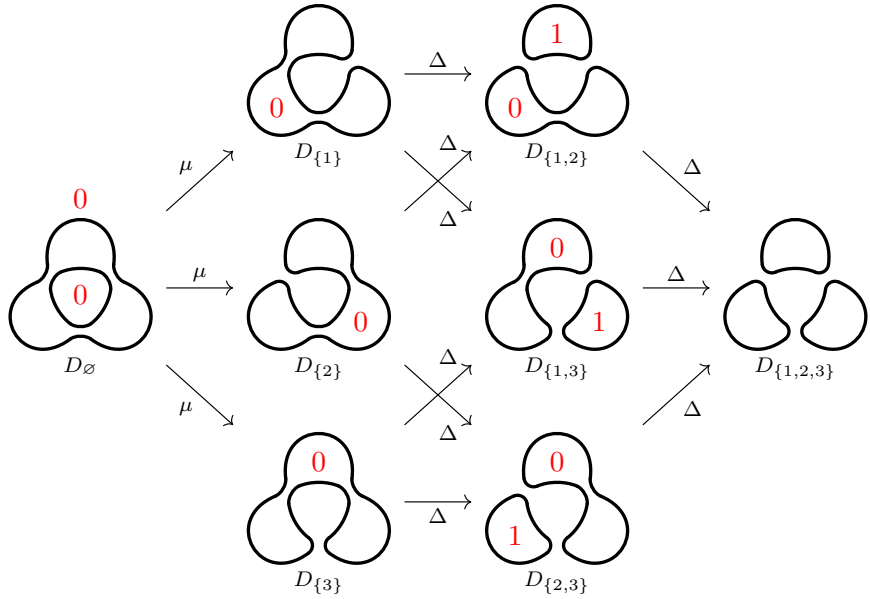
\begin{figure}
\setlength\belowdisplayskip{0pt}
\setlength\belowdisplayshortskip{0pt}
\begin{equation*}
\begin{tikzcd}[row sep=small]
& \underset{D_{\{1\}}}{\pgfpicture\diagKnotTrefoil*{h-smooth,smooth,smooth}[.8]\color{red}\inserttext*[x=-7pt,y=-5pt]{LI}{0}\endpgfpicture} \ar[r,"\Delta"] \ar[dr,"\Delta"' pos=.9] & \underset{D_{\{1,2\}}}{\pgfpicture\diagKnotTrefoil*{h-smooth,h-smooth,smooth}[.8]\color{red}\inserttext*[y=8pt]{UI}{1}\inserttext*[x=-7pt,y=-5pt]{LI}{0}\endpgfpicture} \ar[dr,"\Delta"] & \\
\underset{D_\varnothing}{\pgfpicture\diagKnotTrefoil*{smooth,smooth,smooth}[.8]\color{red}\inserttext*[y=-10pt]{UI}{0}\inserttext*[y=8pt]{UO}{0}\endpgfpicture} \ar[ur,"\mu"] \ar[r,"\mu"] \ar[dr,"\mu"'] & \underset{D_{\{2\}}}{\pgfpicture\diagKnotTrefoil*{smooth,h-smooth,smooth}[.8]\color{red}\inserttext*[x=7pt,y=-5pt]{RI}{0}\endpgfpicture} \ar[ur,"\Delta" pos=.9] \ar[dr,"\Delta"' pos=.9] & \underset{D_{\{1,3\}}}{\pgfpicture\diagKnotTrefoil*{h-smooth,smooth,h-smooth}[.8]\color{red}\inserttext*[x=7pt,y=-5pt]{RI}{1}\inserttext*[y=8pt]{UI}{0}\endpgfpicture} \ar[r,"\Delta"] & \underset{D_{\{1,2,3\}}}{\diagKnotTrefoil*{h-smooth,h-smooth,h-smooth}[.8]} \\
& \underset{D_{\{3\}}}{\pgfpicture\diagKnotTrefoil*{smooth,smooth,h-smooth}[.8]\color{red}\inserttext*[y=8pt]{UI}{0}\endpgfpicture} \ar[ur,"\Delta" pos=.9] \ar[r,"\Delta"'] & \underset{D_{\{2,3\}}}{\pgfpicture\diagKnotTrefoil*{smooth,h-smooth,h-smooth}[.8]\color{red}\inserttext*[x=-7pt,y=-5pt]{LI}{1}\inserttext*[y=8pt]{UI}{0}\endpgfpicture} \ar[ur,"\Delta"'] &
\end{tikzcd}
\end{equation*}
\caption{An ``unfortunate choice'' of twisting weights}
\label{fig:UnfortunateTwistingWeight}
\end{figure}
It is seen that the function $\nu$ cannot extend to $\pi_0(D_{\{1,2,3\}})$ satisfying the condition.

Note that a required function $\nu$ on $D$ \emph{does exist} in the example above.
For example, by swapping the values of the two circles in $D_{\{1,2\}}$ in \cref{fig:UnfortunateTwistingWeight}, one can extend $\nu$ to $\pi_0(D_{\{1,2,3\}})$.
This choice however seems to be \emph{ad-hoc}, and it is uncertain if a similar technique is available for general diagrams having a similar problem.
The observation implies that, if there is an algorithm constructing $\nu$ for arbitrary link diagrams, then we would need finer control of a path in the cube through which $\nu$ is extended.

The plan of paper is as follows.
After preparing basic notations and terminology in \cref{sec:Notation}, we review the statement of Khovanov's assertion and state our main result in \cref{sec:KhovanovProposition}.
The actual construction of the function $\nu$ will be done in \cref{sec:Construction}.

\subsection*{Acknowledgments}
The authors would like to thank Professor Khovanov for valuable discussions and encouraging us to publish this work.
This work was supported by JSPS KAKENHI Grant Number JP25K06999.

\section{Notation and terminology}\label{sec:Notation}

Throughout the note, we will denote by $\mathbb K$ a fixed commutative ring, and tensor products are always taken over $\mathbb K$.
In addition, $A$ always denotes a commutative Frobenius algebra over $\mathbb K$.

A \emph{link diagram} in this paper is unoriented unless otherwise stated.
For a link diagram $D$, we write $\mathcal X(D)$ for the set of crossings of $D$.
A \emph{state} on $D$ is just a subset of $\mathcal X(D)$.
For each state $S\subset\mathcal X(D)$, we construct a $1$-dimensional submanifold $D_S\subset\mathbb R^2$ by replacing each crossing $c\in\mathcal X(D)$ with two disjoint arcs in two different ways according to whether $c\in S$ or not:
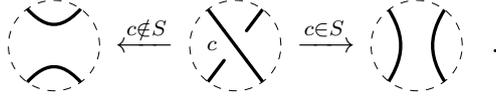
\begin{equation*}
\dashedcircled{.6}{\diagSmooth{H}{}}
\;\xleftarrow{c\notin S}\;
\dashedcircled{.6}{\diagCross*[c]{negative}}
\;\xrightarrow{c \in S}\;
\dashedcircled{.6}{\diagSmooth{V}{}}
\quad.
\end{equation*}
Write $\pi_0(D_S)$ for the set of circles in $D_S$.
If $(S,c)$ is a pair of $S\subset\mathcal X(D)$ and $c\in\mathcal X(D)\setminus S$, then $D_{S\cup\{c\}}$ is obtained from $D_S$ in either of the following ways:
\begin{itemize}
  \item if $c$ is adjacent to only a single circle in $D_S$, say $C\in\pi_0(D_S)$, then $c$ \emph{splits} $C$ into two circles in $D_{S\cup\{c\}}$;
  \item if $c$ is adjacent to two circles in $D_S$, say $C_1,C_2\in\pi_0(D_S)$, then $c$ \emph{merges} $C_1$ and $C_2$ into a single circle in $D_{S\cup\{c\}}$.
\end{itemize}
The circles in $D_S$ apart from $c$ are identified with the corresponding circles in $D_{S\cup\{c\}}$ and often denoted by the same symbols.
\begin{definition}
For a finite set $\Lambda$, a \emph{commutative $\Lambda$-cube} consists of a family $\{V(S)\}_{S\subset\Lambda}$ together with a family of morphisms $\xi_\lambda^S:V(S)\to V(S\cup\{\lambda\})$ for each pair of $S\subset\Lambda$ and $\lambda\in\Lambda\setminus S$ which make the diagram below commute:
\begin{equation}
\label{eq:CommutativeCubeCommutative}
\begin{tikzcd}
V(S) \ar[r,"\xi^S_\lambda"] \ar[d,"\xi^S_\mu"'] & V(S\cup\{\lambda\}) \ar[d,"\xi^{S\cup\{\lambda\}}_\mu"] \\
V(S\cup\{\mu\}) \ar[r,"\xi^{S\cup\{\mu\}}_\lambda"] & V(S\cup\{\lambda,\mu\})
\end{tikzcd}
\end{equation}
\end{definition}

For a commutative $\Lambda$-cube $V=\{V(S)\}_{S\subset\Lambda}$, the structure morphism $\xi^S_\lambda$ is often abbreviated by $\xi_\lambda$; hence, the condition \eqref{eq:CommutativeCubeCommutative} is also written $\xi_\lambda\xi_\mu=\xi_\mu\xi_\lambda$.

For a Frobenius algebra $A$ over $\mathbb K$, and for a link diagram $D$, we define a commutative $\mathcal X(D)$-cube $V_{D;A}$ by
\begin{equation*}
V_{D;A}(S) \coloneqq \bigotimes_{C\in\pi_0(D_S)} A_C
,
\end{equation*}
where $A_C$ is a copy of $A$ associated with a circle $C$ in $D_S$.
The structure morphism $\xi_c=\xi^S_c:V_{D;A}(S)\to V_{D;A}(S\cup\{c\})$ consists of
\begin{itemize}
  \item the comultiplication $\Delta:A_C\to A_{C'_1}\otimes A_{C'_2}$ if $c$ splits $C$ into $C'_1$ and $C'_2$;
  \item the multiplication $\mu:A_{C_1}\otimes A_{C_2}\to A_{C'}$ if $c$ merges $C_1$ and $C_2$ into $C'$;
  \item the identity on $A_C$ if $C$ is apart from $c$.
\end{itemize}

Given a commutative $\Lambda$-cube $V=\{V(S)\}_S$, one can construct a chain complex $C_V=\{C_V^i\}_i$ by $C_V^i\coloneqq\bigoplus_{\lvert S\rvert=i} V(S)$ with the differential $d=\sum\pm\xi_\lambda$, where the sign is chosen in an appropriate manner (see \cite{Khovanov2000} and \cite{BarNatan2002}).
In the case of $V=V_{D;A}$ for a link diagram $D$, we in particular write $C(D;A)\coloneqq C_{V_{D;A}}$.
It turns out that the homology $H^i(D;A)\coloneqq H^i(C^\ast(D;A),d)$ results in a link invariant for some Frobenius algebras $A$ (\cite{Khovanov2000} and \cite{Khovanov2006}).

\section{Twist of Frobenius algebra and Khovanov's proposition}\label{sec:KhovanovProposition}

We are interested in the comparison of the homology $H^i(D;A)$ for different Frobenius algebras $A$.
As a special case, we are reduced to the following.

\begin{definition}
For an invertible element $\theta\in A$, the \emph{twist} of $A$ by $\theta$ is a Frobenius algebra $A^\theta$ which is identical to $A$ as an algebra but with the ``twisted'' coalgebra structure given by
\begin{equation*}
\varepsilon^\theta(x)\coloneqq \varepsilon(\theta x)
,\qquad
\Delta^\theta(x)\coloneqq \Delta(\theta^{-1} x)
.
\end{equation*}
\end{definition}

In \cite[Proposition~3]{Khovanov2006}, the following is asserted.

\begin{proposition}[{\cite[Proposition~3]{Khovanov2006}}]
\label{prop:LinkHomologyFrobeniusTwist}
Let $A$ be a Frobenius algebra and $\theta\in A$ an invertible element.
Then, for every link diagram $D$, the commutative $\mathcal X(D)$-cubes $V_{D;A}$ and $V_{D;A^\theta}$ are isomorphic.
In particular, $C(D;A^\theta)\cong C(D;A)$ as chain complexes.
\end{proposition}

In \cite{Khovanov2006}, Khovanov only gave a sketch of the proof of \cref{prop:LinkHomologyFrobeniusTwist}, though it has a subtle gap in its induction step of the construction of the isomorphism.
The goal of the paper is to give a complete proof for \cref{prop:LinkHomologyFrobeniusTwist}.

We first reduce \cref{prop:LinkHomologyFrobeniusTwist} to a combinatorial problem.

\pushQED\qed
\begin{lemma}
\label{lem:TwistComparison}
For an invertible element $\theta\in A$, the following commute for every $p,q\in\mathbb Z$:
\begin{equation*}
\begin{tikzcd}[baseline=(\tikzcdmatrixname-2-2.south)]
A^\theta \ar[r,"\theta^{p+q-1}"] \ar[d,"\Delta^\theta"'] & A \ar[d,"\Delta"] \\
A^\theta\otimes A^\theta \ar[r,"\theta^p\otimes\theta^q"] & A\otimes A
\end{tikzcd}
\qquad
\begin{tikzcd}[baseline=(\tikzcdmatrixname-2-2.south)]
A^\theta\otimes A^\theta \ar[r,"{\theta^p\otimes\theta^q}"] \ar[d,"{\mu^\theta(=\mu)}"'] & A\otimes A \ar[d,"\mu"] \\
A^\theta \ar[r,"\theta^{p+q}"] & A
\end{tikzcd}
\qedhere
\end{equation*}
\end{lemma}

In view of \cref{lem:TwistComparison}, \cref{prop:LinkHomologyFrobeniusTwist} is proved by taking the following data.

\begin{definition}
\label{def:TwistingWeight}
Let $D$ be a link diagram.
A \emph{twisting weight} on $D$ is a map
\begin{equation*}
\nu:\coprod_{S\subset\mathcal X(D)} \pi_0(D_S)\to \mathbb Z
\end{equation*}
satisfying the following conditions:
\begin{enumerate}[label=\upshape(\alph*)]
  \item if $C\in\pi_0(D_S)$ splits into $C'_1,C'_2\in\pi_0(D_{S\cup\{c\}})$, then $\nu(C)=\nu(C'_1)+\nu(C'_2)-1$;
  \item if $C_1,C_2\in\pi_0(D_S)$ merge into $C'\in\pi_0(D_{S\cup\{c\}})$, then $\nu(C_1)+\nu(C_2)=\nu(C')$;
  \item if $C\in\pi_0(D_S)$ is identified with $C'\in\pi_0(D_{S\cup\{c\}})$, then $\nu(C)=\nu(C')$.
\end{enumerate}
\end{definition}

We can indeed prove \cref{prop:LinkHomologyFrobeniusTwist} for a diagram $D$ with a twisting weight $\nu$.
Namely, for each state $S\subset\mathcal X(D)$, define
\begin{equation*}
\widehat \theta^\nu:V_{D;A^\theta}(S)
= \bigotimes_{C\in\pi_0(D_S)} A^\theta_C
\xrightarrow{\bigotimes_C \theta^{\nu(C)}} \bigotimes_{C\in\pi_0(D_S)} A_C
= V_{D;A}(S)
.
\end{equation*}
By virtue of \cref{lem:TwistComparison}, $\widehat\theta^\nu$ commutes with the structure morphisms, so it defines a morphism of commutative $\mathcal X(D)$-cubes.
Furthermore, since $\theta$ is invertible, it is obvious that $\widehat\theta^\nu$ is in fact an isomorphism.
In other words, \cref{prop:LinkHomologyFrobeniusTwist} follows from the following.

\begin{theorem}
\label{maintheorem}
Every link diagram $D$ admits a twisting weight.
\end{theorem}

\begin{remark}
As a special case, if the invertible element $\theta\in A$ belongs to the coefficient ring $\mathbb K$, then one can prove \cref{prop:LinkHomologyFrobeniusTwist} without taking a twisting weight.
Indeed, in this case, $\theta$ freely passes through the tensor products, so the two commutative squares in \cref{lem:TwistComparison} reduces to the following:
\begin{equation*}
\begin{tikzcd}[column sep=3.5em]
A^\theta \ar[r,"\theta^k\cdot\mathrm{id}"] \ar[d,"\Delta^\theta"'] & A \ar[d,"\Delta"] \\
A^\theta\otimes A^\theta \ar[r,"\theta^{k+1}\cdot\mathrm{id}"] & A\otimes A
\end{tikzcd}
\qquad
\begin{tikzcd}
A^\theta\otimes A^\theta \ar[r,"{\theta^k\cdot\mathrm{id}}"] \ar[d,"{\mu^\theta(=\mu)}"'] & A\otimes A \ar[d,"\mu"] \\
A^\theta \ar[r,"\theta^k\cdot\mathrm{id}"] & A
\end{tikzcd}
\end{equation*}
Hence, one can construct an isomorphism $V_{D;A^\theta}(S)\cong V_{D;A}(S)$ only by counting the number of comultiplication $\Delta$ in the path from the initial state $\varnothing$ to $S$.
\end{remark}

\section{Construction of twisting weight}\label{sec:Construction}

In order to construct a twisting weight on a link diagram, it is convenient to describe the structure of the cube $V_{D;A}$ in terms of \emph{cobordisms}.
Namely, if $S$ is a state on $D$, then for each $c\in\mathcal X(D)\setminus S$, we regard the structure morphism $\xi_c^S$ as a $2$-dimensional cobordism from $D_S$ to $D_{S\cup\{c\}}$ obtained by attaching a saddle around $c$.
It turns out that we still have $\xi_{c_1}\circ\xi_{c_2}=\xi_{c_2}\circ\xi_{c_1}$ as cobordisms; in other words, the family of cobordisms $\{\xi_c\}$ yields a commutative $\mathcal X(D)$-cube in the \emph{cobordism category}.
In fact, the commutative cube $V_{D;A}$ is obtained from this cube by applying the \emph{TQFT} associated with the Frobenius algebra $A$.

For a cobordism $W$, we write
\begin{equation*}
\gamma(W)\coloneqq
\frac{\lvert\pi_0(\boundout W)\rvert - \lvert\pi_0(\boundin W)\rvert - \chi(W)}{2}
.
\end{equation*}
The following is straightforward.

\pushQED\qed
\begin{lemma}\leavevmode
\label{lem:GammaAdditive}
\begin{enumerate}[label=\upshape(\arabic*)]
  \item $\gamma(W_2\circ W_1) = \gamma(W_2) + \gamma(W_1)$.
  \item $\gamma(W\amalg W') = \gamma(W) + \gamma(W')$.
\qedhere
\end{enumerate}
\end{lemma}

\begin{definition}
Let $W:Y_0\to Y_1$ be a $2$-dimensional cobordism.
A pair of functions $\nu_0:\pi_0(Y_0)\to\mathbb Z$ and $\nu_1:\pi_0(Y_1)\to\mathbb Z$ is said to be \emph{compatible with $W$} if, for each connected component $\Sigma\subset W$, we have
\begin{equation*}
\sum_{C\in\pi_0(\boundout\Sigma)} \nu_1(C) = \sum_{C\in\pi_0(\boundin\Sigma)} \nu_0(C) + \gamma(\Sigma).
\end{equation*}
\end{definition}

\begin{lemma}
\label{lem:CompositionCompatible}
Let $W_1:Y_0\to Y_1$ and $W_2:Y_1\to Y_2$ be $2$-dimensional cobordisms and $\nu_i:\pi_0(Y_i)\to\mathbb Z$.
\begin{enumerate}[label=\upshape(\arabic*)]
  \item\label{itm:CompositionCompatible:comp} If $(\nu_0,\nu_1)$ and $(\nu_1,\nu_2)$ are compatible with $W_1$ and $W_2$ respectively, then $(\nu_0,\nu_2)$ is compatible with $W_2\circ W_1$.
  \item\label{itm:CompositionCompatible:descend} If $\pi_0(\boundout W_1)\to \pi_0(W_1)$ is bijective, if $(\nu_0,\nu_1)$ is compatible with $W_1$, then $(\nu_1,\nu_2)$ is compatible with $W_2$ if and only if $(\nu_0,\nu_2)$ is compatible with $W_2\circ W_1$.
\end{enumerate}
\end{lemma}
\begin{proof}
The part~\ref{itm:CompositionCompatible:comp} is immediate from \cref{lem:GammaAdditive}.
To see \ref{itm:CompositionCompatible:descend}, notice that, if $\pi_0(\boundout W_1)\to\pi_0(W_1)$ is bijective, then $\pi_0(W_2)\to\pi_0(W_2\circ W_1)$ is also bijective.
Thus, each connected component $\Sigma$ of $W_2\circ W_1$ is of the form $\Sigma^{(2)}\circ(\amalg^{(1)}_i \Sigma'_i)$ for connected components $\Sigma^{(2)}\in\pi_0(W_2)$ and $\Sigma^{(1)}_i\in\pi_0(W_1)$.
If $(\nu_0,\nu_1)$ is compatible with $W_1$, then we obtain
\begin{gather*}
\sum_{C\in\pi_0(\boundout\Sigma)} \nu_2(C)
= \sum_{C\in\pi_0(\boundout\Sigma^{(2)})}\nu_2(C)
,\\[1ex]
\begin{split}
\sum_{C\in\pi_0(\boundin\Sigma)} \nu_0(C)
&= \sum_i\sum_{C\in\pi_0(\boundin\Sigma^{(1)}_i)} \nu_0(C) \\
&= \sum_i\sum_{C'\in\pi_0(\boundout\Sigma^{(1)}_i)} \nu_1(C') - \gamma(\Sigma^{(1)}_i) \\
&= \sum_{C'\in\pi_0(\boundin\Sigma^{(2)})} \nu_1(C') - \gamma({\textstyle\coprod_i}\Sigma^{(1)}_i)
.
\end{split}
\end{gather*}
The result is now obvious.
\end{proof}

If $\Sigma$ is a cobordism with a single saddle, then we have
\begin{equation*}
\gamma(\Sigma)=\begin{cases*}
1 & if $\Sigma$ splits a circle,\\
0 & if $\Sigma$ merges a pair of circles.
\end{cases*}
\end{equation*}
This computation leads to the following.

\pushQED\qed
\begin{lemma}
\label{lem:WeightIffCompatible}
Let $D$ be a link diagram.
A set of functions $\{\nu_S:\pi_0(D_S)\to\mathbb Z\}_S$ forms a twisting weight if and only if each pair of the form $(\nu_S,\nu_{S\cup\{c\}})$ for $c\notin S$ is compatible with $\xi^S_c$.
\qedhere
\end{lemma}

\begin{definition}
Let $D$ be a link diagram and $S$ a state.
\begin{enumerate}[label=\upshape(\arabic*)]
\item A pair of crossings $c_1$ and $c_2$ of $D$ are said to be \emph{disjoint} in $S$ if there is no circle in $D_S$ adjacent to both of them.
\item A pair of crossings $c_1$ and $c_2$ of $D$ are said to be \emph{parallel} in $S$ if they are adjacent to the same pair of distinct circles in $D_S$.
\item A triple of crossings $c_1$, $c_2$, and $c_3$ of $D$ are said to form a \emph{triangle} in $S$ if they are in the position as in \cref{fig:ParallelTriangle}.%
\end{enumerate}
\end{definition}
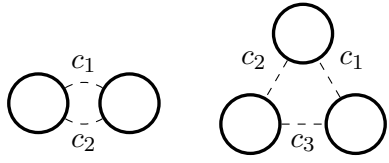
\begin{figure}
\setlength\belowdisplayskip{0pt}
\setlength\belowdisplayshortskip{0pt}
\begin{equation*}
\begin{tikzpicture}
\node[circle,draw,normal edge] (A) at (-.6,0) {\strut};
\node[circle,draw,normal edge] (B) at (.6,0) {\strut};
\draw[dashed]
    (A) to[out=30,in=150] node[above]{$c_1$} (B)
    (A) to[out=-30,in=210] node[below]{$c_2$} (B)
    ;
\end{tikzpicture}
\qquad
\begin{tikzpicture}
\node[circle,draw,normal edge] (A) at (90:.8) {\strut};
\node[circle,draw,normal edge] (B) at (210:.8) {\strut};
\node[circle,draw,normal edge] (C) at (330:.8) {\strut};
\draw[dashed] (A) -- node[above left]{$c_2$} (B) -- node[below]{$c_3$} (C) -- node[above right]{$c_1$} (A);
\end{tikzpicture}
\end{equation*}
\caption{Parallel crossings and triangle}
\label{fig:ParallelTriangle}
\end{figure}

\begin{lemma}
\label{lem:NonGParallel}
Let $D$ be a link diagram such that $D_\varnothing$ is connected.
Suppose $S$ is a state on $D$ such that, for every $c\in S$, $\xi_c:D_{S\setminus\{c\}}\to D_S$ splits a circle into two.
Then $S$ has at least one pair of crossings which are not parallel.
Furthermore, if there is no disjoint pair of crossings in $S$, then $S$ has no triangle.
\end{lemma}
\begin{proof}
If every pair of crossing in $S$ are mutually parallel, then the crossings in $S$ are depicted as in \cref{fig:GlobalParallel}.%
\begin{figure}
\setlength\belowdisplayskip{0pt}
\setlength\belowdisplayshortskip{0pt}
\begin{equation*}
\begin{tikzpicture}
\node[ellipse,draw,normal edge,inner xsep=7,inner ysep=20] (A) at (-.6,0) {};
\node[ellipse,draw,normal edge,inner xsep=7,inner ysep=20] (B) at (.6,0) {};
\pgftext[y=-2pt]{$\scriptstyle\vdots$}
\draw[dashed]
  (A.60) to[out=15,in=165] (B.120)
  (A.30) to[bend left] (B.150)
  (A.-60) to[out=-15,in=195] (B.240)
  ;
\end{tikzpicture}
\end{equation*}
\caption{All parallel crossings}
\label{fig:GlobalParallel}
\end{figure}
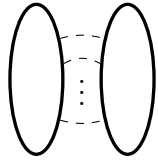
Since there are no other crossings, one can easily see that $D_\varnothing$ is not connected.

We next suppose that $S$ has no disjoint pair of crossings.
If $S$ admits a triangle, say $c_1$, $c_2$, and $c_3$ as in \cref{fig:ParallelTriangle}, then each circle in the picture cannot have crossings in $S$ in the ``interior'' of the circle.
Hence, the triangle becomes like in \cref{fig:TriangleInitial} in the state $S\setminus\{a,b,c\}$ such that no crossing in $S\setminus\{a,b,c\}$ intersects with the red triangle.
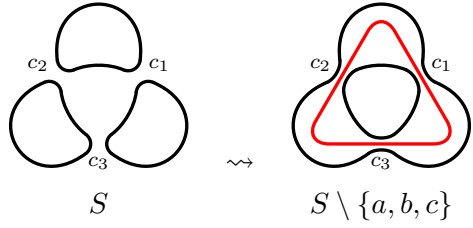
\begin{figure}
\setlength\belowdisplayskip{0pt}
\setlength\belowdisplayshortskip{0pt}
\begin{equation*}
\underset{\displaystyle\rule{0pt}{3ex}S}{\diagKnotTrefoil*[c_1,c_2,c_3]{h-smooth,h-smooth,h-smooth}}
\quad\leadsto\quad
\underset{\displaystyle\rule{0pt}{3ex}S\setminus\{a,b,c\}}{%
\begin{tikzpicture}
\diagKnotTrefoil*[c_1,c_2,c_3]{smooth,smooth,smooth}
\draw[red,rounded corners=10] (90:1.2) -- (210:1.2) -- (330:1.2) -- cycle;
\end{tikzpicture}
}
\end{equation*}
\caption{Triangle in the initial state $\varnothing$}
\label{fig:TriangleInitial}
\end{figure}
In particular, $D_\varnothing$ is not connected.
\end{proof}

\begin{lemma}
\label{lem:ExtWellDefined}
Let $S$ be a state on a link diagram $D$ and $a,b\in S$.
Suppose that we are given a functions
\begin{gather*}
\nu_{00}:\pi_0(D_{S\setminus\{a,b\}})\to\mathbb Z
,\\
\nu_{10}:\pi_0(D_{S\setminus\{b\}})\to\mathbb Z
,\\
\nu_{01}:\pi_0(D_{S\setminus\{a\}})\to\mathbb Z
\end{gather*}
such that the pairs $(\nu_{00},\nu_{10})$ and $(\nu_{00},\nu_{01})$ are compatible with $\xi^{S\setminus\{a,b\}}_a:D_{S\setminus\{a,b\}}\to D_{S\setminus\{b\}}$ and $\xi^{S\setminus\{a,b\}}_b:D_{S\setminus\{a,b\}}\to D_{S\setminus\{a\}}$ respectively.
For $C\in\pi_0(D_S)$, let $\Sigma_a\in\pi_0(\xi^{S\setminus\{a\}}_a)$ and $\Sigma_b\in\pi_0(\xi^{S\setminus\{b\}}_b)$ be the connected components containing $C$.
If $\boundout\Sigma_a=\boundout\Sigma_b=C$, then we have
\begin{equation*}
\sum_{C'\in\pi_0(\boundin\Sigma_b)} \nu_{10}(C')
= \sum_{C'\in\pi_0(\boundin\Sigma_a)} \nu_{01}(C')
.
\end{equation*}
\end{lemma}
\begin{proof}
We denote by $\widetilde\Sigma\subset\xi^{S\setminus\{b\}}_b\circ\xi^{S\setminus\{a,b\}}_a=\xi^{S\setminus\{a\}}_a\circ\xi^{S\setminus\{a,b\}}_b$ the connected component containing $C$.
We assert that $\boundout\widetilde\Sigma=C$.
Indeed, if either of $a$ and $b$ is not adjacent to $C$ in $S$, then $\widetilde\Sigma$ is identified either of $\Sigma_a$ or $\Sigma_b$, so the assertion is obvious.
If both $a$ and $b$ are adjacent to $C$ in $S$, then the possible configurations of these crossings are as in \cref{fig:MergeConfig}.%
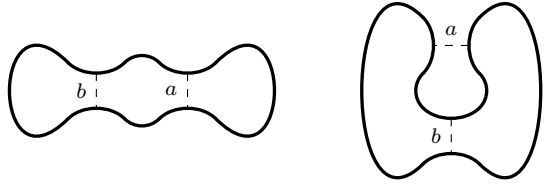
\begin{figure}
\setlength\belowdisplayskip{0pt}
\setlength\belowdisplayshortskip{0pt}
\begin{equation*}
\begin{tikzpicture}[baseline=-.5ex]
\useasboundingbox (-2,-.7) rectangle (2,.7);
\node[crossing,cross type=h-smooth,cross angle=90,inner sep=15pt] (B) at (-.6,0) {};
\node[crossing,cross type=h-smooth,cross angle=90,inner sep=15pt] (A) at (.6,0) {};
\draw[normal edge]
  (B.out left) to[out=135,in=225,looseness=5] (B.in left)
  (B.out right) to[out=45,in=135] (A.out left)
  (A.out right) to[out=45,in=315,looseness=5] (A.in right)
  (A.in left) to[out=225,in=315] (B.in right)
  ;
\draw[dashed]
  (A.north) -- node[left]{$\scriptstyle a$}(A.south)
  (B.north) -- node[left]{$\scriptstyle b$}(B.south)
  ;
\end{tikzpicture}
\qquad
\begin{tikzpicture}[baseline=-.5ex]
\useasboundingbox (-1.3,-1.3) rectangle (1.3,1.3);
\node[crossing,cross type=smooth,cross angle=90,inner sep=15pt] (A) at (0,.6) {};
\node[crossing,cross type=h-smooth,cross angle=90,inner sep=15pt] (B) at (0,-.6) {};
\draw[normal edge]
  (A.out left) to[out=135,in=225,looseness=2] (B.in left)
  (B.out left) to[out=135,in=225] (A.in left)
  (A.out right) to[out=45,in=315,looseness=2] (B.in right)
  (B.out right) to[out=45,in=315] (A.in right)
  ;
\draw[dashed]
  (A.east) -- node[above]{$\scriptstyle a$}(A.west)
  (B.north) -- node[left]{$\scriptstyle b$}(B.south)
  ;
\end{tikzpicture}
\end{equation*}
\caption{Configuration of ``merging'' crossings}
\label{fig:MergeConfig}
\end{figure}
In both cases, the assertion is verified by the direct computation.
Notice that, by virtue of the assertion, we can write $\widetilde\Sigma=\Sigma_b\circ W_a = \Sigma_a\circ W_b$ for $W_a\subset\xi^{S\setminus\{a,b\}}_a$ and $W_b\subset\xi^{S\setminus\{a,b\}}_b$.
Since $\gamma(\Sigma_a)=\gamma(\Sigma_b)=0$, we have $\gamma(W_a)=\gamma(W_b)=\gamma(\widetilde\Sigma)$.
Now, since $(\nu_{00},\nu_{10})$ and $(\nu_{00},\nu_{01})$ are compatible with $\xi^{S\setminus\{a,b\}}_a$ and $\xi^{S\setminus\{a,b\}}_b$ respectively, we have
\begin{equation*}
\sum_{C'\in\pi_0(\boundin\Sigma_b)} \nu_{10}(C')
= \sum_{C''\in\pi_0(\boundin\widetilde\Sigma)} \nu_{00}(C'') + \gamma(\widetilde\Sigma)
= \sum_{C'\in\pi_0(\boundin\Sigma_a)} \nu_{01}(C')
,
\end{equation*}
which completes the proof.
\end{proof}

\begin{lemma}
\label{lem:TargetCircle}
Let $D$ be a link diagram.
For every state $S$, exactly one of the following holds:
\begin{enumerate}[label=\upshape(\arabic*)]
  \item for each circle component $C\in\pi_0(D_S)$, there is a crossing $c\in S$ such that $\xi^{S\setminus\{c\}}_c:D_{S\setminus\{c\}}\to D_S$ contains a connected component $\Sigma$ with $\boundout\Sigma=C$;
  \item there is a circle component $C_0\in\pi_0(D_S)$ such that every crossing $c\in S$ is adjacent to $C_0$ and one other circle in $D_S$.
\end{enumerate}
\end{lemma}
\begin{proof}
If there is a crossing $c\in S$ such that $\xi^{S\setminus\{c\}}_c:D_{S\setminus\{c\}}\to D_S$ merges circles, then every circle in $D_S$ is a unique target of a connected component of $\xi^{S\setminus\{c\}}_c$.
Also, every circle in $D_S$ which is not adjacent to $c$ is the unique target of a connected component of $\xi^{S\setminus c}_c:D_{S\setminus\{c\}}\to D_S$.
Therefore, if $D_S$ has a circle which is not a unique target of any connected components of $\xi^{S\setminus c}_c$ for any crossings $c$, then every such circle must be adjacent to all the crossings.
\end{proof}

\begin{proposition}
\label{prop:WeightConnectedInit}
If $D$ is a link diagram such that $D_\varnothing$ is connected, then $D$ admits a twisting weight.
\end{proposition}
\begin{proof}
In view of \cref{lem:WeightIffCompatible}, we construct a function $\nu_S:\pi_0(D_S)\to\mathbb Z$ for each state $S$ by induction on the cardinality $\lvert S\rvert$ so that $(\nu_S,\nu_{S\cup\{c\}})$ is compatible with the cobordism $\xi^S_c$ for each pair of a state $S$ and a crossing $c$ with $c\notin S$.
We set $\nu_\varnothing\equiv 0$.
For each crossing $c$, take $\nu_{\{c\}}:\pi_0(D_{\{c\}})\to\mathbb Z$ arbitrarily so that $(\nu_\varnothing,\nu_{\{c\}})$ is compatible with $\xi^\varnothing_c$.
For $S$ with $\lvert S\rvert\ge 2$, we proceed the induction by splitting it to the following two cases in view of \cref{lem:TargetCircle}.

First, suppose that every circle $C\in\pi_0(D_S)$ admits a crossing $c \in S$ such that $\xi^{S\setminus\{c\}}_c$ contains a connected component $\Sigma_C$ with $\boundout\Sigma_C=C$.
We define $\nu(C)\in\mathbb Z$ by
\begin{equation}\label{eq:DefnuS}
\nu_S(C)\coloneqq\sum_{C'\in\pi_0(\boundin\Sigma_C)} \nu_{S\setminus\{c\}}(C')
.
\end{equation}
Note that, by virtue of \cref{lem:ExtWellDefined}, $\nu(C)$ does not depend on the choice of $c$.
We show that $(\nu_{S\setminus\{c\}},\nu_S)$ is compatible with $\xi^{S\setminus\{c\}}_c$ for each $c\in S$.
If $c$ merges circles, then the assertion directly follows from the independence mentioned above.
Suppose $c$ splits a circle $C'\in\pi_0(D_{S\setminus\{c\}})$ into $C_1,C_2\in\pi_0(D_S)$.
If a connected component $\Sigma\subset\xi^{S\setminus\{c\}}_c$ is not adjacent to $c$, then $\Sigma$ is by definition the identity cobordism,
and hence $\nu_S(\boundout\Sigma)=\nu_{S\setminus\{c\}}(\boundin\Sigma)$ by definition of $\nu_S$.
On the other hand, let us denote by $\Sigma_c$ the connected component of $\xi^{S\setminus\{c\}}_c$ containing the saddle point.
If there is a crossing $c'\in S$ disjoint to $c$, then the circles $C'$, $C_1$, and $C_2$ are identified with those in $\pi_0(D_{S\setminus\{c,c'\}})$ and $\pi_0(D_{S\setminus \{c'\}})$ respectively.
Similarly, $\Sigma_c$ is also seen as a connected component of the cobordism $\xi^{S\setminus\{c,c'\}}_c$.
The induction hypothesis then implies
\begin{equation*}
\nu_S(C_1)+\nu_S(C_2)
= \nu_{S\setminus\{c'\}}(C_1) + \nu_{S\setminus\{c'\}}(C_2)
= \nu_{S\setminus\{c,c'\}}(C') + 1
= \nu_{S\setminus\{c\}}(C') + 1
.
\end{equation*}
Hence $(\nu_{S\setminus\{c\}},\nu_S)$ is compatible with $\xi^{S\setminus\{c\}}_c$ in this case.
Also, if there is a crossing $c'\in S$ which merges circles, then it turns out that the cobordism $\xi^{S\setminus\{c,c'\}}_{c'}:D_{S\setminus\{c,c'\}}\to D_{S\setminus\{c\}}$ also merges circles.
According to \cref{lem:CompositionCompatible}, in this case, $(\nu_{S\setminus\{c\}},\nu_S)$ is compatible with $\xi^{S\setminus\{c\}}_c$ if and only if $(\nu_{S\setminus\{c,c'\}},\nu_S)$ is compatible with the composition $\xi^{S\setminus\{c\}}_c\circ\xi^{S\setminus\{c,c'\}}_{c'}=\xi^{S\setminus\{c'\}}_{c'}\circ\xi^{S\setminus\{c,c'\}}_c$.
By induction hypothesis, $(\nu_{S\setminus\{c,c'\}},\nu_{S\setminus\{c'\}})$ is compatible with $\xi^{S\setminus\{c,c'\}}_c$ while $(\nu_{S\setminus\{c'\}},\nu_S)$ is compatible with $\xi^{S\setminus\{c'\}}_{c'}$ by definition.
Hence the assertion is verified.
Suppose now that $c$ admits no disjoint crossings in $S$ and that all the crossings cause splitting of circles.
Take crossings $c_i\in S$ for $i=1,2$ so that $c_i$ is adjacent only to the circle $C_i$.
By virtue of \cref{lem:NonGParallel}, the configuration of $c$, $c_1$, and $c_2$ in $D_S$ must be as in \cref{fig:CrossLinearConfig}.%
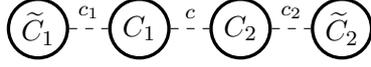
\begin{figure}
\setlength\belowdisplayskip{0pt}
\setlength\belowdisplayshortskip{0pt}
\begin{equation*}
\begin{tikzpicture}
\node[normal edge,circle,draw] (C1) at (-.666,0) {\strut};
\node[normal edge,circle,draw] (C2) at (.666,0) {\strut};
\node[normal edge,circle,draw] (A) at (-2,0) {\strut};
\node[normal edge,circle,draw] (B) at (2,0) {\strut};
\pgftext[at={\pgfpointanchor{C1}{center}}]{$C_1$}
\pgftext[at={\pgfpointanchor{C2}{center}}]{$C_2$}
\pgftext[at={\pgfpointanchor{A}{center}}]{$\widetilde C_1$}
\pgftext[at={\pgfpointanchor{B}{center}}]{$\widetilde C_2$}
\draw[dashed]
  (A)
  -- node[above]{$\scriptstyle c_1$} (C1)
  -- node[above]{$\scriptstyle c$} (C2)
  -- node[above]{$\scriptstyle c_2$} (B)
  ;
\end{tikzpicture}
\end{equation*}
\caption{Linear configuration of three crossings}
\label{fig:CrossLinearConfig}
\end{figure}
Let $\widetilde C_1$ and $\widetilde C_2$ be circles as in \cref{fig:CrossLinearConfig}.
We also write $\widetilde C'_i$ for $i=1,2$ for the circle in $D_{S\setminus\{c_i\}}$ which splits into $\widetilde C_i$ and $C_i$.
By induction hypothesis, we then have
\begin{gather*}
\begin{aligned}
\nu_{S\setminus\{c\}}(\widetilde C_1) &=\nu_{S\setminus\{c_2\}}(\widetilde C_1)
,\quad&
\nu_{S\setminus\{c\}}(\widetilde C_2) &=\nu_{S\setminus\{c_1\}}(\widetilde C_2)
,\\
\nu_S(C_1) &=\nu_{S\setminus\{c_2\}}(C_1)
,\quad&
\nu_S(C_2) &=\nu_{S\setminus\{c_1\}}(C_2)
,
\end{aligned}
\\
\nu_{S\setminus\{c_2\}}(\widetilde C_1)+\nu_{S\setminus\{c_2\}}(C_1)
= \nu_{S\setminus\{c_1,c_2\}}(\widetilde C'_1) + 1
,\\
\nu_{S\setminus\{c_1\}}(C_2)+\nu_{S\setminus\{c_1\}}(\widetilde C_2)
= \nu_{S\setminus\{c_1,c_2\}}(\widetilde C'_2) + 1
,\\
\nu_{S\setminus\{c_1,c_2\}}(\widetilde C'_1) + \nu_{S\setminus\{c_1,c_2\}}(\widetilde C'_2)
= \nu_{S\setminus\{c,c_1,c_2\}}(\widetilde C') + 1
,
\end{gather*}
where $\widetilde C'$ is the circle in $D_{S\setminus\{c,c_1,c_2\}}$ formed by the four circles.
In addition, the induction hypothesis also implies that $(\nu_{S\setminus\{c,c_1,c_2\}},\nu_{S\setminus\{c\}})$ is compatible with the composition of the cobordisms $\xi^{S\setminus\{c,c_1\}}_{c_1}\circ\xi^{S\setminus\{c,c_1,c_2\}}_{c_2}$, so
\begin{equation*}
\nu_{S\setminus\{c\}}(\widetilde C_1) + \nu_{S\setminus\{c\}}(C') + \nu_{S\setminus\{c\}}(\widetilde C_2)
= \nu_{S\setminus\{c,c_1,c_2\}}(\widetilde C') + 2
.
\end{equation*}
Combining the equations above, one obtains $\nu_S(C_1)+\nu_S(C_2)=\nu_{S\setminus\{c\}}(C')+1$.
This proves that $(\nu_S,\nu_{S\setminus\{c\}})$ is compatible with $\xi^{S\setminus\{c\}}_c$.

Second, suppose that there is a circle $C_0$ in $D_S$ such that every crossing $c\in S$ is adjacent to $C_0$ and a one other circle.
Note that if there is another such circle, then every crossing in $S$ is adjacent to both of the two circles, which is impossible by virtue of \cref{lem:NonGParallel}.
It follows that every circle $C$ other than $C_0$ admits a crossing $c$ which is not adjacent to $C$; in this case, we may identify $C$ with the corresponding circle in $\pi_0(D_{S\setminus\{c\}})$.
We then define $\nu_S(C)$ by
\begin{equation*}
\nu_S(C)\coloneqq
\begin{cases*}
\gamma(W_S) - \sum_{C\neq C_0}\nu_S(C) & if $C=C_0$,\\
\nu_{S\setminus\{c\}}(C) & if $C\neq C_0$ and if $C$ is not adjacent to $c$,
\end{cases*}
\end{equation*}
where $W_S$ is the cobordism obtained by composing the chain of cobordisms connecting $D_\varnothing$ to $D_S$.
We verify that $(\nu_{S\setminus\{c\}},\nu_S)$ is compatible with the cobordism $\xi^{S\setminus\{c\}}_s:D_{S\setminus\{c\}}\to D_S$ for each crossing $c$.
For a connected component $\Sigma\subset\xi^{S\setminus\{c\}}_c$, if $\Sigma$ is not adjacent to $c$, then $\Sigma$ is the identity cobordism so that $\nu_S(\boundout\Sigma)=\nu_{S\setminus\{c\}}(\boundin\Sigma)$.
Otherwise, $\Sigma$ is the saddle cobordism from $C' $ to $C_0\amalg C_1$ for $C'\in\pi_0(D_{S\setminus\{c\}})$ and $C_1\in\pi_0(D_S)$ with $C_1\neq C_0$.
Thus, we have
\begin{equation*}
\begin{split}
\nu_S(C_0)+\nu_S(C_1)
&= \gamma(W_S) - \sum_{C\in\pi_0(D_S)\setminus\{C_0,C_1\}} \nu_S(C) \\[1ex]
&= \gamma(\xi^{S\setminus\{c\}}_c) + \gamma(W_{S\setminus\{c\}}) - \sum_{C\in\pi_0(D_{S\setminus\{c\}})\setminus\{C'\}} \nu_{S\setminus\{c\}}(C) \\[1ex]
&= \gamma(\xi^{S\setminus\{c\}}_c) + \nu_{S\setminus\{c\}}(C') + \sum_{C''\in\pi_0(\boundin W_{S\setminus\{c\}})} \nu_\varnothing(C'') \\[1ex]
&= \gamma(\xi^{S\setminus\{c\}}_c) + \nu_{S\setminus\{c\}}(C')
,
\end{split}
\end{equation*}
which completes the proof.
\end{proof}

The general case reduces to the case of \cref{prop:WeightConnectedInit} by virtue of the following result.

\begin{lemma}
\label{lem:WeightCrossChange}
Suppose $D_+$ and $D_-$ are link diagrams which differ from one another only by the type of a single crossing.
Then $D_+$ admits a twisting weight if and only if $D_-$ does.
\end{lemma}
\begin{proof}
We suppose $D_+$ admits a twisting weight $\nu$ and construct a twisting weight $\nu'$ on $D_-$; the other way is similar.
Let $c_0$ be the crossing whose over- and under-arc are switched in $D_+$ and in $D_-$.
Choosing an edge $e$ in $D_+$ adjacent to $c_0$, we define a function $\phi:\coprod_{S\subset\mathcal X(D_-)}\pi_0(D_{-,S})\to \mathbb Z$ by
\begin{equation*}
\phi(C)\coloneqq\begin{cases*}
1 & if $C\in\pi_0(D_{-,S})$ with $c_0\in S$ and $e\subset C$,\\
0 & otherwise.
\end{cases*}
\end{equation*}
Note that each circle in a resolution of $D_-$ is identified with one in a resolution of $D_+$.
From this point of view, define a function $\nu':\coprod_{S\subset\mathcal X(D_-)}\pi_0(D_{-,S})\to\mathbb Z$ by
\begin{equation*}
\nu'(C)\coloneqq \nu(C) - \phi(C).
\end{equation*}
We show that $\nu'$ is indeed a twisting weight.
Suppose $S$ is a state on $D_-$ and $c\notin S$.
We check the conditions in the following two cases.

\paragraph{Case of $c\ne c_0$:}
First notice that, if a circle $C$ in $D_{-,S}$ is not adjacent to $c$, and if $C'$ is the corresponding circle in $D_{-,S\cup\{c\}}$, then since the edges of $C$ is exactly those of $C'$, we have
\begin{equation*}
\nu'(C)-\nu'(C')=\nu(C)-\phi(C)-\nu(C')+\phi(C') = 0
.
\end{equation*}
Hence, we have to verify the condition on the circles adjacent to $c$.
If $c$ splits a circle $C$ in $D_{-,S}$ into $C_1'$ and $C_2'$ in $D_{-,S\cup\{c\}}$, then we have
\begin{align*}
\nu'(C) - \nu'(C_1') - \nu'(C_2') + 1
&= \nu(C) - \phi(C) - \nu(C_1') + \phi(C_1') - \nu(C_2') + \phi(C_2') + 1 \\
&= -\phi(C) + \phi(C_1') + \phi(C_2')
\end{align*}
since $c$ also splits $C$ into $C_1'$ and $C_2'$ in $D_+$.
Observing the set of edges of $C$ is exactly the disjoint union of those of $C_1'$ and $C_2'$, one obtains $\nu'(C)=\nu'(C_1')+\nu'(C_2')-1$.
If $c$ merges circles $C_1$ and $C_2$ in $D_{-,S}$ into $C'\in D_{-,S\cup\{c\}}$, then
\begin{align*}
\nu'(C_1)+\nu'(C_2)-\nu'(C')
&= \nu(C_1)-\phi(C_1)+\nu(C_2)-\phi(C_2)-\nu(C')+\phi(C') \\
&= -\phi(C_1)-\phi(C_2)+\phi(C').
\end{align*}
As with the ``split'' case, the set of edges of $C'$ is exactly the disjoint union of those of $C_1$ and $C_2$, so we obtain $\nu'(C_1)+\nu'(C_2)=\nu'(C')$.

\paragraph{Case of $c=c_0$:}
As in the preceding case, it is enough to verify the condition on circles adjacent to $c_0$.
Note that if $c_0$ splits a circle $C$ in $D_{-,S}$ into $C_1'$ and $C_2'$ $D_{-,S\cup\{c_0\}}$, then it merges $C_1'$ and $C_2'$ into $C$ in the cube of $D_+$.
Furthermore, since the edge $e$ is contained in exactly one of $C_1'$ or $C_2'$, and since $\phi(C)=0$ by definition, we obtain
\begin{align*}
\nu'(C)-\nu'(C_1')-\nu'(C_2')+1
&= \nu(C) - \phi(C) - \nu(C_1') + \phi(C_1') - \nu(C_2') + \phi(C_2') + 1 \\
&= - \phi(C) - \phi(C_1') - \phi(C_2') + 1 \\
&= 0
.
\end{align*}
The condition is also verified similarly in the case of ``merge.''
\end{proof}

\begin{proof}[Proof of \cref{maintheorem}]
Let $D$ be an arbitrary link diagram.
Notice that, if $D$ is a disjoint union of diagrams $D^{(1)},\dots, D^{(n)}$, then a twisting weight on $D$ is determined by its restriction to each $D^{(k)}$.
Hence, we may assume $D$ is connected without loss of generality.
Notice that, in this case, there is a state $S\subset\chi(D)$ for which $D_S$ is connected.
Indeed, since $D$ is connected, every adjacent pair of circles in $D_S$ are connected by a crossing $c\in\mathcal X(D)$, so they are merged in either of the state $S\cup\{c\}$ or $S\setminus\{c\}$.
Iterating the process, one reaches to a required state.
If $S$ is such a state, then perform crossing change on each crossing $c\in S$; let us denote by $D'$ the resulting diagram.
By construction, the smoothing $D'_\varnothing$ is identified with $D_S$ and hence connected.
Therefore, $D'$ admits a twisting weight by \cref{prop:WeightConnectedInit}.
Furthermore, since $D$ and $D'$ are connected by a sequence of crossing change, so the twisting weight on $D'$ is transferred to $D$ by virtue of \cref{lem:WeightCrossChange}.
This completes the proof.
\end{proof}

\bibliographystyle{plain}
\bibliography{reference.bib}

\begin{thebibliography}{1}

\bibitem{BarNatan2002}
D.~Bar-Natan.
\newblock On {K}hovanov's categorification of the {J}ones polynomial.
\newblock {\em Algebraic \& Geometric Topology}, 2(1):337--370, 2002.

\bibitem{BarNatan2005}
D.~Bar-Natan.
\newblock {K}hovanov's homology for tangles and cobordisms.
\newblock {\em Geometry \& Topology}, 9(3):1443--1499, 2005.

\bibitem{Khovanov2000}
M.~Khovanov.
\newblock A categorification of the {J}ones polynomial.
\newblock {\em Duke Mathematical Journal}, 101(3):359--426, February 2000.

\bibitem{Khovanov2006}
M.~Khovanov.
\newblock Link homology and {F}robenius extensions.
\newblock {\em Fundamenta Mathematicae}, 190:179--190, 2006.

\end{thebibliography}

\end{document}